\documentclass[12pt]{amsart}
\usepackage{graphicx}
\usepackage{amsfonts}
\usepackage{amssymb}
\usepackage{mathrsfs}
\usepackage[cp1250]{inputenc}
\newtheorem{theorem}{Theorem}[section]

\newtheorem{corollary}[theorem]{Corollary}
\newtheorem{lemma}[theorem]{Lemma}
\theoremstyle{definition}

\numberwithin{equation}{section}

\begin{document}
\title{Some characterizations of asymmetric truncated Toeplitz operators}
\author{Bartosz {\L}anucha, Ma{\l}gorzata Michalska}

\address{
Bartosz {\L}anucha,  \newline Institute of Mathematics,
\newline Maria Curie-Sk{\l}odowska University, \newline pl. M.
Curie-Sk{\l}odowskiej 1, \newline 20-031 Lublin, Poland}
\email{bartosz.lanucha@poczta.umcs.lublin.pl}

\address{
Ma{\l}gorzata Michalska,  \newline Institute of Mathematics,
\newline Maria Curie-Sk{\l}odowska University, \newline pl. M.
Curie-Sk{\l}odowskiej 1, \newline 20-031 Lublin, Poland}
\email{malgorzata.michalska@poczta.umcs.lublin.pl}

\date{\today}
\subjclass[2010]{47B32, 47B35, 30H10.}
\keywords{model space, truncated
Toeplitz operator, asymmetric truncated
Toeplitz operator}
\begin{abstract}
It was recently proved that in some special cases asymmetric truncated Toeplitz operators can be characterized in terms of compressed shifts and rank-two operators of special form. In this paper we show that such characterizations hold in all cases.
\end{abstract}
\maketitle

\baselineskip1.4\baselineskip

\section{Introduction}
As usual, let $H^2$ denote the classical Hardy space. The space $H^2$ can be seen as a space of functions analytic in the unit disk
$\mathbb{D}=\{z:|z|<1\}$ or as a closed
subspace of $L^2:=L^2(\partial\mathbb{D})$. In the first case $H^2$ consists of functions analytic in $\mathbb{D}$ with square summable MacLaurin coefficients and in the second it consists of functions from $L^2$ such that their Fourier coefficients with negative indices vanish.

The unilateral shift $S$ on $H^2$ is the operator of multiplication by the independent variable, that is,
  $$Sf(z)=z\cdot f(z).$$
The adjoint of $S^*$ of $S$ is called the backward shift. A simple verification shows that
  $$S^*f(z)=\frac{f(z)-f(0)}{z}.$$

The famous Beurling theorem provides a characterization of all $S^*$-invariant subspaces of $H^2$. Namely, a closed nontrivial subspace of $H^2$ is $S^*$-invariant if and only if it is of the form
  $$K_\alpha=H^2\ominus \alpha H^2,$$
where $\alpha$ is an inner function, i.e., $\alpha$ belongs to the algebra $H^{\infty}$ of bounded analytic functions and $|\alpha|=1$ a.e. on $\partial
\mathbb{D}$. The space $K_{\alpha}$ is called the model space associated with $\alpha$.

The operators $S$ and $S^*$ are two examples of classical Toeplitz operators. Let $P$ denote the orthogonal projection from $L^2$ onto $H^2$. A Toeplitz operator $T_{\varphi}$ with symbol $\varphi\in L^\infty$ is defined on $H^2$ by
  $$T_{\varphi}f=P(\varphi f).$$
Clearly, if $\varphi\in H^\infty$, then $T_\varphi$ is just the multiplication by $\varphi$. Also, the operator $T_{\varphi}$ is densely defined on bounded functions whenever $\varphi\in L^2$, and extends to a bounded operator on $H^2$ if and only if $\varphi \in L^\infty$. We have $S=T_z$ and $S^{*}=T_{\overline{z}}$.

Truncated Toeplitz operators are compressions of Toeplitz operators to model spaces. More precisely, a truncated Toeplitz operator $A_{\varphi}^{\alpha}$ with a symbol $\varphi\in L^2$ is defined on the model space $K_{\alpha}$ by
  $$A_{\varphi}^{\alpha}f=P_{\alpha}(\varphi f),$$
where $P_{\alpha}$ is the orthogonal projection from $L^2$ onto $K_{\alpha}$. In particular, $S_\alpha=A^\alpha_z$ is called the compressed shift. Since $K_\alpha$ is $S^*$-invariant, it easily follows that $S^*_\alpha=S^*_{|K_\alpha}$.

The study of the class of truncated Toeplitz operators was started in 2007 with D. Sarason's paper \cite{s} (see \cite{gar3}). Recently, the authors in \cite{ptak} and \cite{part2,part} initiated the study of so-called asymmetric truncated Toeplitz operators (see also \cite{blicharz1,blicharz2}).

Let $\alpha$, $\beta$ be two inner functions. An asymmetric
truncated Toeplitz operator $A_{\varphi}^{\alpha,\beta}$ with a symbol $\varphi\in
L^2$ is the
operator from $K_{\alpha}$ into $K_{\beta}$ given by
  $$A_{\varphi}^{\alpha,\beta}f=P_{\beta}(\varphi f).$$
Clearly, $A_{\varphi}^{\alpha}=A_{\varphi}^{\alpha,\alpha}$.
Put
  $$\mathscr{T}(\alpha,\beta)=\{A_{\varphi}^{\alpha,\beta}\colon\, \varphi\in
  L^2\ \mathrm{and}\ A_{\varphi}^{\alpha,\beta}\
  \mathrm{is\ bounded}\}.$$

A bounded linear operator $T$ on $H^2$ is a Toeplitz operator if and only if $T-S^* TS=0$. D. Sarason \cite{s} proved that a bounded linear operator $A$ on $K_\alpha$ is a truncated Toeplitz operator if and only if $A-S_{\alpha}^* AS_\alpha$ is an operator of a special kind and rank at most two. Similar characterization for the operators from $\mathscr{T}(\alpha,\beta)$ were proved for the case when $\beta$ divides $\alpha$ (that is, $\alpha/\beta$ is an inner function) \cite{ptak}, and for the case when $\alpha$ and $\beta$ are finite Blaschke products (in other words, when $K_\alpha$ and $K_\beta$ are finitely dimensional) \cite{L2}.

In this paper we provide such characterizations of the operators from $\mathscr{T}(\alpha,\beta)$ for all $\alpha, \beta$.

\section{Characterizations of asymmetric truncated Toeplitz operators}
Recall that  a model space $K_{\alpha}$ is a reproducing kernel Hilbert space. That is to say that for every $f$ in the model space $K_{\alpha}$ and each $w\in\mathbb{D}$,
  $$f(w)=\langle f, k_{w}^{\alpha}\rangle,$$
where the reproducing kernel function $k_{w}^{\alpha}$ is of the form
  $$k_{w}^{\alpha}(z)=\frac{1-\overline{\alpha(w)}\alpha(z)}{1-\overline{w}z}.$$
  Observe that since $k_{w}^{\alpha}\in H^{\infty}$, the set $K_{\alpha}^{\infty}=K_{\alpha}\cap H^{\infty}$ is dense in $K_{\alpha}$.

A conjugate kernel is the function
$\widetilde{k}_{w}^{\alpha}=C_{\alpha}{k}_{w}^{\alpha}$, where $C_{\alpha}:L^2\to L^2$ is given by
  $$C_{\alpha}f(z)=\widetilde{f}(z)=\alpha(z)\overline{z}\overline{f(z)},\quad |z|=1.$$
It can be seen that $C_\alpha$, as defined on $L^2$, is an antilinear isometric involution (a map with such property is called a conjugation). It can also be verified that the conjugation $C_\alpha$ preserves $K_\alpha$. Therefore $\widetilde{k}_{w}^{\alpha}\in K_\alpha$ for all $w\in\mathbb{D}$ and a simple computation gives
  $$\widetilde{k}_{w}^{\alpha}(z)=\frac{\alpha(z)-\alpha(w)}{z-w}.$$

\begin{theorem}\label{thm_ATTO_rank_two_char}
Let $A$ be a bounded linear operator from $K_\alpha$ into $K_\beta$. Then $A\in\mathscr{T(\alpha,\beta)}$ if and only if there exist $\psi\in K_\beta$ and $\chi\in K_\alpha$ such that
\begin{align}\label{eq_ATTO_rank_two_char}
  A-S_\beta A S^*_\alpha=\psi\otimes k^\alpha_0 +k^\beta_0\otimes \chi.
\end{align}
\end{theorem}
The proof of Theorem \ref{thm_ATTO_rank_two_char} follows the proof given by D. Sarason for truncated Toeplitz operators \cite[Theorem 4.1]{s} and requires some auxiliary lemmas.
\begin{lemma}\label{lem_ATTO_rank2_char_aux1}
For every $\varphi\in L^2$ the equality
\begin{align*}
  A^{\alpha,\beta}_\varphi-S_\beta A^{\alpha,\beta}_\varphi S^*_\alpha=\psi\otimes k^\alpha_0 +k^\beta_0\otimes \chi
\end{align*}
holds on $H^2$, where
\begin{align*}
  &\psi=S_\beta P_\beta\left(\overline{z}\varphi\right) \in K_\beta,\\
  &\chi=P_\alpha\left(\overline{\varphi}\right) \in K_\alpha.
\end{align*}
\end{lemma}

\begin{proof}
Let $f\in K_\alpha^\infty$ and $g\in K_\beta^\infty$. The functions
  $$S_\alpha^* f(z)=S^* f(z) =\frac{f(z)-f(0)}{z}$$
and
  $$S_\beta^* g(z)=S^* g(z)=\frac{g(z)-g(0)}{z}$$
clearly belong to $H^2\cap L^\infty=H^\infty$. Therefore,

\begin{align*}
  \left\langle S_\beta A_\varphi^{\alpha,\beta}S_\alpha^* f,g\right\rangle
  &=\left\langle P_\beta(\varphi S_\alpha^*f),S_\beta^* g\right\rangle
  =\left\langle \varphi \cdot\frac{f-f(0)}{z},S_\beta^* g\right\rangle\\
  &=\left\langle \overline{z}\varphi f,S_\beta^* g\right\rangle
  -f(0)\left\langle \overline{z}\varphi ,S_\beta^* g\right\rangle\\
  &=\left\langle \overline{z}\varphi f,\frac{g-g(0)}{z}\right\rangle
  -\left\langle f,k^\alpha_0 \right\rangle
  \left\langle S_\beta P_\beta(\overline{z}\varphi), g\right\rangle\\
  &=\left\langle \overline{z}\varphi f,\overline{z}g\right\rangle
  -\overline{g(0)}\left\langle \overline{z}\varphi f,\overline{z}\right\rangle
  -\left\langle (S_\beta P_\beta(\overline{z}\varphi)\otimes k^\alpha_0)f, g\right\rangle\\
  &=\left\langle A^{\alpha,\beta}_\varphi f, g\right\rangle
  -\langle k^\beta_0,g \rangle\left\langle f,P_\alpha(\overline{\varphi})\right\rangle
  -\left\langle (S_\beta P_\beta(\overline{z}\varphi)\otimes k^\alpha_0)f, g\right\rangle\\
  &=\left\langle A^{\alpha,\beta}_\varphi f, g\right\rangle
  -\langle (k^\beta_0\otimes P_\alpha(\overline{\varphi}))f,g\rangle
  -\left\langle (S_\beta P_\beta(\overline{z}\varphi)\otimes k^\alpha_0)f, g\right\rangle.
\end{align*}
From this
\begin{align}\label{eq_pr_ATTO_rank2_char_1}
   A^{\alpha,\beta}_\varphi f -S_\beta A^{\alpha,\beta}_\varphi S^*_\alpha f
   =(S_\beta P_\beta(\overline{z}\varphi)\otimes k^\alpha_0)f+(k^\beta_0\otimes P_\alpha(\overline{\varphi}))f \quad \text{for all } f\in K_\alpha^\infty.
\end{align}
For an arbitrary $f\in K_\alpha$ there exists $\{f_n\}\in K_\alpha^\infty$ such that $f_n\to f$ in $H^2$ norm and both $\varphi f_n\to \varphi f$ and $\varphi S^*_\alpha f_n\to\varphi S^*_\alpha f$ in $L^1$. Hence, $A_\varphi^{\alpha,\beta} f_n\to A_\varphi^{\alpha,\beta} f$ and $A_\varphi^{\alpha,\beta} S^*_\alpha f_n\to A_\varphi^{\alpha,\beta} S^*_\alpha f$ uniformly on compact subsets of $\mathbb{D}$, which implies that \eqref{eq_pr_ATTO_rank2_char_1} holds for all $f\in K_\alpha$.
\end{proof}

\begin{lemma}\label{lem_ATTO_rank2_aux}
If $\varphi=\overline{\chi}+\psi$, where $\chi\in K_\alpha$ and $\psi\in K_\beta$, $\psi(0)=0$, then the equality
\begin{align*}
  \langle A_\varphi^{\alpha,\beta}f,g\rangle=\sum_{n=1}^\infty\langle(S^n_\beta \psi\otimes S^n_\alpha k^\alpha_0 +S^n_\beta k^\beta_0\otimes S^n_\alpha \chi)f,g\rangle
\end{align*}
holds for all $f\in K_\alpha^\infty$ and $g\in K_\beta^\infty$.
\end{lemma}
\begin{proof}
Note that if $\varphi=\overline{\chi}+\psi$, where $\chi\in K_\alpha$ and $\psi\in K_\beta$, $\psi(0)=0$, then
  $$P_\alpha(\overline{\varphi})=\chi \qquad \text{and }\qquad S_\beta P_\beta(\overline{z}\varphi)
  =S_\beta P_\beta(\overline{z}\psi)=S_\beta S^*_\beta \psi=\psi.$$
By Lemma \ref{lem_ATTO_rank2_char_aux1},
  $$A^{\alpha,\beta}_\varphi -S_\beta A^{\alpha,\beta}_\varphi S^*_\alpha
   =\psi\otimes k^\alpha_0+k^\beta_0\otimes \chi,$$
and for every $n\geq 0$,
\begin{align*}
  &S^n_\beta A_\varphi^{\alpha,\beta} (S^*_\alpha)^n -S^{n+1}_\beta A_\varphi^{\alpha,\beta} (S^*_\alpha)^{n+1}
  =S^n_\beta \psi\otimes S^n_\alpha k^\alpha_0+S^n_\beta k^\beta_0\otimes S^n_\alpha \chi.
\end{align*}
From this
  $$A_\varphi^{\alpha,\beta}
  =\sum_{n=1}^N (S^n_\beta \psi\otimes S^n_\alpha k^\alpha_0
  +S^n_\beta k^\beta_0\otimes S^n_\alpha \chi)
  +S^{N+1}_\beta A_\varphi^{\alpha,\beta} (S^*_\alpha)^{N+1}$$
    for every $N\geq 1$.

    If $f\in K_\alpha^\infty$ and $g\in K_\beta^\infty$, then
\begin{displaymath}
\begin{split}
  \left\langle S^{N+1}_\beta\right.& \left. A_\varphi^{\alpha,\beta}(S^*_\alpha)^{N+1} f,g\right\rangle\\
  &=\left\langle A_{\overline{\chi}}^{\alpha,\beta}(S^*_\alpha)^{N+1} f,
  (S^*_\beta)^{N+1} g\right\rangle
  +\left\langle A_\psi^{\alpha,\beta}(S^*_\alpha)^{N+1} f,(S^*_\beta)^{N+1} g\right\rangle\\
  &=\left\langle T_{\overline{\chi}}(S^*)^{N+1} f,
  (S^*)^{N+1} g\right\rangle
  +\left\langle (S^*)^{N+1} f,T_{\overline{\psi}}(S^*)^{N+1} g\right\rangle\\
  &=\left\langle P(\overline{z}^{N+1}\overline{\chi}f) ,
   (S^*)^{N+1}g\right\rangle
  +\left\langle  (S^*)^{N+1}f, P(\overline{z}^{N+1}\overline{\psi}g)\right\rangle \to 0\quad \text{as } N\to 0
  \end{split}
\end{displaymath}
since $$\|P(\overline{z}^{N+1}\overline{\chi}f)\|_2\leq \|f\|_{\infty}\cdot\|\chi\|_2,\qquad \|P(\overline{z}^{N+1}\overline{\psi}g)\|_2\leq \|g\|_{\infty}\cdot\|\psi\|_2$$ and $(S^*)^N\to 0$ in the strong operator topology. Therefore
  $$\langle A_\varphi^{\alpha,\beta}f,g\rangle
  =\sum_{n=1}^\infty \langle (S^n_\beta \psi\otimes S^n_\alpha k^\alpha_0
  +S^n_\beta k^\beta_0\otimes S^n_\alpha \chi)f,g\rangle$$
  for all $f\in K_{\alpha}^{\infty}$ and $g\in K_{\beta}^{\infty}$.

\end{proof}

\begin{proof}[Proof of Theorem \ref{thm_ATTO_rank_two_char}]
If $A\in\mathscr{T}(\alpha,\beta)$, then $A=A_\varphi^{\alpha,\beta}$ for some $\varphi\in L^2$ and it satisfies \eqref{eq_ATTO_rank_two_char} by Lemma \ref{lem_ATTO_rank2_char_aux1}.

Assume now that $A$ is a bounded linear operator from $K_\alpha$ into $K_\beta$ such that \eqref{eq_ATTO_rank_two_char} holds for $\psi\in K_\beta$ and $\chi\in K_\alpha$. Without any loss of generality we can assume that $\psi(0)=0$. Indeed, if this was not the case we would replace $\psi$ and $\chi$ with $\psi-c k^\beta_0$ and $\chi+\overline{c} k^\alpha_0$, respectively, for $c=\psi(0)/(1-|\beta(0)|^2)$.

Define $\varphi=\overline{\chi}+\psi$. By Lemma \ref{lem_ATTO_rank2_aux}, for every $f\in K_{\alpha}^{\infty}$ and $g\in K_{\beta}^{\infty}$,
  $$\langle A_\varphi^{\alpha,\beta}f,g\rangle
  =\sum_{n=1}^\infty\langle (S^n_\beta \psi\otimes S^n_\alpha k^\alpha_0
  +S^n_\beta k^\beta_0\otimes S^n_\alpha \chi)f,g\rangle.$$
But an argument similar to the one given in the proof of Lemma \ref{lem_ATTO_rank2_char_aux1} shows that
  $$\langle Af,g\rangle=\sum_{n=1}^\infty \langle (S^n_\beta \psi\otimes S^n_\alpha k^\alpha_0
  +S^n_\beta k^\beta_0\otimes S^n_\alpha \chi)f,g\rangle,$$
and so $A=A_\varphi^{\alpha,\beta}\in \mathscr{T}(\alpha,\beta)$.
\end{proof}

\begin{corollary}\label{cor_ATTO_rank2}
Let $A$ be a bounded linear operator form $K_\alpha$ into $K_\beta$. Then $A\in\mathscr{T}(\alpha,\beta)$ if and only if
\begin{equation}\label{eq_ATTO_rank_two_char2}
A-S^*_\beta A S_\alpha=\psi\otimes \widetilde{k}^\alpha_0 +\widetilde{k}^\beta_0\otimes \chi
\end{equation}
for some $\psi\in K_\beta$ and $\chi\in K_\alpha$.
\end{corollary}

\begin{proof}

Recall first that the linear operator $A$ belongs to $\mathscr{T}(\alpha,\beta)$ if and only if $B=C_{\beta}AC_{\alpha}$ belongs to $\mathscr{T}(\alpha,\beta)$ (see \cite[p. 9]{L2}).

If $A\in\mathscr{T}(\alpha,\beta)$, then $C_{\beta}AC_{\alpha}\in\mathscr{T}(\alpha,\beta)$. By Theorem \ref{thm_ATTO_rank_two_char}, there exist functions $\chi_0\in K_{\alpha}$ and $\psi_0\in K_{\beta}$ such that
\begin{equation*}
C_{\beta}AC_{\alpha}-S_{\beta}C_{\beta}AC_{\alpha} S_{\alpha}^{*}=\psi_0\otimes k_{0}^{\alpha}+k_{0}^{\beta}\otimes\chi_0.
\end{equation*}
Thus (using the symmetry of compressed shifts) we get

\begin{displaymath}
\begin{split}
A-S_{\beta}^{*}A S_{\alpha}&=C_{\beta}^2AC_{\alpha}^2-C_{\beta}S_{\beta}C_{\beta}A C_{\alpha}S_{\alpha}^{*}C_{\alpha}\\
&=C_{\beta}(\psi_0\otimes k_{0}^{\alpha}+k_{0}^{\beta}\otimes\chi_0)C_{\alpha}=\widetilde{\psi}_0\otimes \widetilde{k}_{0}^{\alpha}+\widetilde{k}_{0}^{\beta}\otimes\widetilde{\chi}_0,
\end{split}
\end{displaymath}
and $A$ satisfies \eqref{eq_ATTO_rank_two_char2} with
$$\psi=\widetilde{\psi}_0\quad\mathrm{and}\quad\chi=\widetilde{\chi}_0.$$

The other implication can be proved in a similar way.
\end{proof}

Conditions from Corollary \ref{cor_ATTO_rank2} can also be formulated in terms of modified compressed shifts.

\begin{corollary}\label{cor_ATTO_rank2_MCS}
Let $A$ be a bounded linear operator form $K_\alpha$ into $K_\beta$. Then $A\in\mathscr{T}(\alpha,\beta)$ if and only if one (and all) of the following conditions holds
\begin{enumerate}
  \item[(a)] $A-S_{\beta,b} A S^*_{\alpha,a}=\psi\otimes k^\alpha_0 +k^\beta_0\otimes \chi$;
  \item[(b)] $A-S^*_{\beta,b} A S_{\alpha,a}=\psi\otimes \widetilde{k}^\alpha_0 +\widetilde{k}^\beta_0\otimes \chi$;
\end{enumerate}
for some $a,b\in\mathbb{C}$, $\psi\in K_\beta$ and $\chi\in K_\alpha$ (possibly different for different conditions), where
  $$S_{\alpha,a}=S_\alpha+a (k^\alpha_0\otimes \widetilde{k}^\alpha_0),\qquad
  S_{\beta,b}=S_\beta+ b(k^\beta_0\otimes \widetilde{k}^\beta_0)$$
are the modified compressed shifts.
\end{corollary}

\begin{proof}
The proof uses Corollary \ref{cor_ATTO_rank2} and is analogous to the proof given in \cite[Theorem 7.1]{s}. The details are therefore left to the reader.
\end{proof}

\begin{corollary}
$\mathscr{T}(\alpha,\beta)$ is closed in the weak operator topology.
\end{corollary}
\begin{proof}
See \cite{s}, \cite{ptak} or \cite{L2}.
\end{proof}

Using Corollary \ref{cor_ATTO_rank2} we can now characterize the operators from $\mathscr{T}(\alpha,\beta)$ in terms of shift invariance. The following corollary was first noted in \cite{ptak} (for the case when $\beta$ divides $\alpha$).

\begin{corollary}\label{cor_ATTO_rank2_SI}
Let $A$ be a bounded linear operator form $K_\alpha$ into $K_\beta$. Then $A\in\mathscr{T}(\alpha,\beta)$ if and only if it has the following property
$$\left\langle AS f, S g\right\rangle=\left\langle A f,g \right\rangle$$
 for all $f\in K_\alpha$, $g\in K_\beta$ such that $Sf\in K_\alpha$, $Sg\in K_\beta$.
\end{corollary}
\begin{proof}
See \cite{L2}, \cite{ptak}.
\end{proof}

\end{document}